\documentclass[12pt]{amsproc}

\theoremstyle{definition}

\theoremstyle{remark}

\numberwithin{equation}{section}
\usepackage{graphicx}
\usepackage{latexsym}
\usepackage{amsmath}
\usepackage{amssymb}
\usepackage{amsfonts}
\usepackage{verbatim}
\usepackage{mathrsfs}
\usepackage{cite}
\usepackage{geometry}
\usepackage[latin1]{inputenc}
\usepackage{color}
\usepackage[colorlinks,citecolor=blue,urlcolor=blue]{hyperref}
\usepackage{float}
\usepackage{subfigure}
\allowdisplaybreaks \allowdisplaybreaks[4]

\newtheorem{tm}{Theorem}[section]
\newtheorem{rk}{Remark}[section]
\newtheorem{df}{Definition}[section]
\newtheorem{ap}{Assumption}[section]
\newtheorem{prop}{Proposition}[section]

\newtheorem{cor}{Corollary}[section]

\newcommand{\E}{\mathbb E}

\newcommand{\N}{\mathbb N}
\newcommand{\R}{\mathbb R}

\begin{document}

\title[]
{Symplectic Runge--Kutta Methods for Hamiltonian Systems Driven by Gaussian Rough Paths}

\author[]{Jialin Hong}
\address{Academy of Mathematics and Systems Science, Chinese Academy of Sciences, Beijing, China}
\curraddr{}
\email{hjl@lsec.cc.ac.cn}
\thanks{Authors are supported by National Natural Science Foundation of China (NO. 91530118, NO. 91130003, NO. 11021101 and NO. 11290142).}
\author[]{Chuying Huang}
\address{Academy of Mathematics and Systems Science, Chinese Academy of Sciences, Beijing, China}
\curraddr{}
\email{huangchuying@lsec.cc.ac.cn (Corresponding author)}
\thanks{}
\author[]{Xu Wang}
\address{Academy of Mathematics and Systems Science, Chinese Academy of Sciences, Beijing, China}
\curraddr{}
\email{wangxu@lsec.cc.ac.cn}
\thanks{}



\date{\today}

\dedicatory{}

\begin{abstract}
We consider Hamiltonian systems driven by multi-dimensional Gaussian processes in rough path sense, which include fractional Brownian motions with Hurst parameter $H\in(1/4,1/2]$. We indicate that the phase flow preserves the symplectic structure almost surely and this property could be inherited by symplectic Runge--Kutta methods, which are implicit methods in general. If the vector fields belong to $Lip^{\gamma}$, we obtain the solvability of Runge--Kutta methods and the pathwise convergence rates. For linear and skew symmetric vector fields, we focus on the midpoint scheme to give corresponding results. Numerical experiments verify our theoretical analysis.
\end{abstract}

\maketitle

\section{Introduction}

We consider a stochastic differential equation (SDE)
\begin{align}\label{sde}
dY_{t}&=V(Y_{t})dX_{t}(\omega),\quad  t\in [0,T]
\end{align}
driven by multi-dimensional Gaussian process in rough path sense.
Under proper assumptions, $X$ can be lifted to a rough path almost surely and then (\ref{sde}) can be interpreted as a rough differential equation (RDE). For example, the lift of the fractional Brownian motion (fBm) with Hurst parameter $H>1/4$ is constructed by piecewise linear approximations and their iterated integrals. We refer to \cite{Friz,MHFriz} and references therein for more details.

The well-posedness of the RDE is given originally in \cite{Lyons} for the case that the vector fields belong to $Lip^{\gamma}$ (in the sense of E. Stein), i.e., they are bounded and smooth enough with bounded derivatives. When the vector fields are linear, the well-posedness result still holds (see e.g., \cite{Friz}). In particular, the solution is equivalent to that in Stratonovich sense almost surely when the noise is a semi-martingale. The robustness of the solution allows numerous researches to be developed, such as the density and ergodicity of SDEs driven by non-Markovian noises (see \cite{Density,MH} and references therein) and the theory of stochastic partial differential equations (see e.g., \cite{MHFriz}).

The author in \cite{Davie} develops an approximation approach to deal with the equation driven by non-differentiable paths, which are further investigated and called step-$N$ Euler schemes in \cite{Friz}. When the equation is driven by fBms with $H>1/2$, modified Euler schemes and Taylor schemes are analyzed in \cite{HuEuler,HuTaylor}. For rougher noises, the step-$N$ Euler schemes are not implementable from the numerical point of view, because they depend on iterated integrals of the noise, which are rather difficult to be simulated. The current implementable methods are the simplified step-$N$ Euler schemes, which are first proposed in \cite{Deya} for fBm with $H\in (1/3,1/2]$. Authors in \cite{Wzk} give the estimate of the pathwise convergence rate of Wong--Zakai approximations for Gaussian processes under proper assumptions and obtain convergence rates of the simplified step-$N$ Euler schemes. The results are applicable for fBm with $H\in (1/4,1/2]$. Further, the $L^{r}$-convergence rate is obtained in \cite{MLMC16} to reduce the complexity of multilevel Monte Carlo.

From the perspective of modeling, Gaussian noises with nontrivial correlations are more general than standard Brownian motions. For a dynamic system with additional nonconservative forces in \cite{Gold}, which could be generally assumed as Gaussian noises, its pathwise movement rule is described by Hamiltonian equations in rough sense.
In this paper, we are concerned about a Hamiltonian system driven by multi-dimensional Gaussian process under certain assumption. The fBm with $H\in (1/4,1/2]$ satisfies our assumption with $\rho=\frac{1}{2H}$, where $\rho$ is a parameter related to the regularity of the noise.
We prove that the rough Hamiltonian system, as a generalization of the classical case in \cite{GeometricEH} and the stochastic case driven by standard Brownian motions in \cite{Milstein,additive}, also has the characteristic property of the Hamiltonian system that its phase flow preserves the symplectic structure almost surely. The symplecticity is naturally considered to be inherited by numerical methods. However, the simplified step-$N$ Euler schemes, which are explicit, are failed to possess discrete symplectic structure. To our best knowledge, there has been few work about implicit methods for RDEs driven by noises rougher than standard Brownian motions.

We make use of symplectic Runge--Kutta methods to preserve the symplectic structure of the original rough Hamiltonian system. Since these methods are implicit methods in general, the solvability of the numerical methods should be taken into consideration, which is one of differences between explicit methods and implicit methods. For SDEs driven by Brownian motions, authors in \cite{Milstein} truncate each increment of Brownian motions to avoid the possibility that the increments could become unbounded, and give the solvability of implicit methods and convergence rates in mean square sense. However, this truncation skill is not applicable here, as the convergence is considered in pathwise sense. Based on the Brouwer's theorem, we prove that if the vector fields belong to $Lip^{\gamma}$, then Runge--Kutta methods are solvable for any time step. If the vector fields are linear and skew symmetric, we get the solvability of the $1$-stage symplectic Runge--Kutta method---the midpoint scheme.

Furthermore, we analyze pathwise convergence rates for symplectic Runge--Kutta methods. For the $Lip^{\gamma}$ case, we obtain that the convergence rate of the midpoint scheme is $(\frac{3}{2\rho}-1-\varepsilon)$, for arbitrary small $\varepsilon$ and $\rho\in[1,3/2)$. The convergence rate of another two symplectic Runge-Kutta methods, whose stages are higher, is $(\frac{1}{\rho}-\frac{1}{2}-\varepsilon)$ for $\rho\in[1,2)$. This is limited by the convergence rate of Wong--Zakai approximations, similar to the simplified step-$3$ Euler scheme in \cite{Wzk}. If the vector fields are linear and skew symmetric, we prove the uniform boundedness of numerical solutions for the midpoint scheme and have that its convergence rate is also $(\frac{3}{2\rho}-1-\varepsilon)$ when $\rho\in[1,3/2)$.

This paper is organized as follows. In Section \ref{sec2}, we introduce notations and definitions in rough path theory as well as the result of the well-posedness of RDEs, which contains both the $Lip^{\gamma}$ case and the linear case. In Section \ref{sec3}, we derive Hamiltonian equations with additional nonconservative forces through the variational principle to obtain the rough Hamiltonian system and prove that the phase flow preserves the symplectic structure almost surely. In Section \ref{sec4}, we propose symplectic Runge--Kutta methods for rough Hamiltonian systems and give the solvability of them. In Section \ref{sec5}, we analyze pathwise convergence rates for symplectic Runge--Kutta methods. Numerical experiments in Section \ref{sec6} verify our theoretical analysis.

\section{Preliminaries in rough path theory}\label{sec2}

We first recall some notations in rough path theory. For $p \in [1,\infty)$,
we are interested in a continuous map $\textbf{X}:[0,T]\rightarrow G^{[p]}(\R^{d})$, where $[p]$ is an integer satisfying $p-1< [p]\leq p$ and $G^{[p]}(\R^{d})$ is the free step-$[p]$ nilpotent Lie group of $\R ^{d}$, equipped with the Carnot-Carath\'{e}odory metric $\rm d$.

Since the group $G^{[p]}(\R^{d})$  is embedded in the truncated step-$[p]$ tensor algebra, i.e., $G^{[p]}(\R^{d})\subset \bigoplus _{n=0}^{[p]}(\R^{d})^{\otimes n}$ with $(\R^{d})^{\otimes 0}=\R$, the increment of $\textbf{X}$ is defined by $\textbf{X}_{s,t}:=\textbf{X}_{s}^{-1}\otimes\textbf{X}_{t}$. Denoting by $\pi_{1}(G^{[p]}(\R^{d}))$ the projection of $G^{[p]}(\R^{d})$ into $\R^{d}$, we have $\pi_{1}(\textbf{X}_{s,t})=\pi_{1}(\textbf{X}_{t})-\pi_{1}(\textbf{X}_{s})$.

Define the $p$-variation of $\textbf{X}$ by
\begin{align*}
\|\textbf{X}\|_{p\text{-}var;[s,t]}:=\sup_{(t_{k}) \in \mathcal{D}([s,t])}\left(\sum_{t_{k}} {\rm d}( \textbf{X}_{t_{k}},\textbf{X}_{t_{k+1}})^{p}\right)^{1/p},
\end{align*}
where $\mathcal{D}([s,t])$ is the set of dissections of $[s,t]$. We call $\textbf{X}$ a weak geometric $p$-rough path if $$\textbf{X}\in\mathcal{C}^{p\text{-}var}([0,T],G^{[p]}(\R^{d})):=\{\textbf{X}:\|\textbf{X}\|_{p\text{-}var;[0,T]}<\infty\},$$
where the set $\mathcal{C}^{p\text{-}var}([0,T],G^{[p]}(\R^{d}))$ contains drivers of RDEs.
In addition, we say $\textbf{X}$ is of H$\rm \ddot{o}$lder-type if
\begin{align*}
\ \|\textbf{X}\|_{1/p\text{-} {\rm H\ddot{o}l};[0,T]}:=\sup_{0\leq s<t\leq T}\dfrac{{\rm d}(\textbf{X}_{s},\textbf{X}_{t})}{|t-s|^{1/p}} < \infty.
\end{align*}
This implies $|\pi_{1}(\textbf{X}_{s,t})|\leq \|\textbf{X}\|_{1/p\text{-}{\rm H\ddot{o}l};[0,T]}|t-s|^{1/p}$ for any $0\leq s<t\leq T$.
Similar to the $1/p$-H$\rm \ddot{o}$lder continuity in classical case, a larger $p$ implies lower regularity of $\textbf{X}$.

The solution of an RDE is constructed by means of a sequence of bounded variation functions on $\R^{d}$, which are appropriate approximations of $\textbf{X}$. Let $x:[0,T]\rightarrow\R^{d}$ be a bounded variation function, we consider its canonical lift to a weak geometric $p$-rough path defined by $S_{[p]}(x)$ with
$$S_{[p]}(x)_{t}:=(1,\int_{0\leq u_{1}\leq t}dx_{u_{1}},...,\int_{0\leq u_{1}<\cdot\cdot\cdot<u_{[p]}\leq t}dx_{u_{1}}\otimes\cdot\cdot\cdot \otimes dx_{u_{[p]}})\in G^{[p]}(\R^{d}).$$
One can observe that $S_{[p]}(x)$ contains information about iterated integrals of $x$ up to $[p]$-level, which is corresponding to the regularity of $\textbf{X}$. If $S_{[p]}(x)$ is close to $\textbf{X}$ in $p$-variation sense, then $x$ is a proper approximation of $\textbf{X}$. In other words, the rougher $\textbf{X}$ is the more information is required.

We now introduce the definition of the solution of an RDE
\begin{align}\label{rde}
dY_{t}=V(Y_{t})d\textbf{X}_{t},\quad Y_{0}=z\in \R^{m}.
\end{align}

\begin{df}(\cite[Definition 10.17]{Friz})\label{solution}
Let $p\in [1,\infty)$ and $\textbf{X} \in \mathcal{C}^{p\text{-}var}([0,T],G^{[p]}(\R^{d}))$. Suppose there exists a sequence of bounded variation functions $\{x^{n}\}$ on $\R^{d}$, such that
\begin{equation*}\label{xn}
\begin{split}
\sup_{n\in \N}\|S_{[p]}(x^{n})\|_{p\text{-}var;[0,T]} < \infty,\\
\lim_{n\rightarrow \infty} \sup_{0\leq s<t\leq T} {\rm d}(S_{[p]}(x^{n})_{s,t},\textbf{X}_{s,t})=0.
\end{split}
\end{equation*}
Suppose in addition $\{y^{n}\}$ are solutions of  equations $dy^{n}_{t}=V(y^{n}_{t})dx^{n}_{t}$, $y^{n}_{0}=z$, in Riemann-Stieltjes integral sense. If $y^{n}_{t}$ converges to $Y_{t}$ in $L^{\infty}([0,T])$-norm, then we call $Y_{t}$ a solution of (\ref{rde}).
\end{df}

To ensure the well-posedness of an RDE, proper assumptions are given for $V$, which will be described by the notation $Lip^{\gamma}$. Throughout this paper, $|\cdot|$ is the Euclidean norm, and we will use $C$ as generic constants, which may be different from line to line.

\begin{df}
Let $\gamma>0$ and $\lfloor \gamma \rfloor$ be the largest integer strictly smaller than $\gamma$, i.e., $\gamma-1\leq\lfloor \gamma \rfloor<\gamma$. We say $V \in Lip^{\gamma}$, if $V$ is $\lfloor \gamma \rfloor$-Fr$\rm\acute{e}$chet differentiable and the $k$th-derivative of $V$, $D^{k}V$, satisfies that
\begin{equation*}
\begin{split}
&|D^{k}V(y)|\leq C,\ \forall k=0,..., \lfloor \gamma \rfloor,\quad \forall\ y \in \R^{m},\\
&|D^{\lfloor \gamma \rfloor }V(y_{1})-D^{\lfloor \gamma \rfloor }V(y_{1})|\leq C |y_{1}-y_{2}|^{\gamma - \lfloor \gamma \rfloor},\quad \forall\ y_{1},y_{2}\in \R^{m},
\end{split}
\end{equation*}
for some constant $C$.
The smallest constant $C$ satisfying these two inequalities is denoted by $|V|_{Lip^{\gamma}}$.
\end{df}

In the sequel, we give the theorem of the well-posedness of an RDE.
\begin{tm}\label{well}
Let  $\textbf{X}=(\textbf{X}^{i})_{1\leq i\leq d} \in \mathcal{C}^{p\text{-}var}([0,T],G^{[p]}(\R^{d}))$. If $V=(V_{i})_{1\leq i\leq d}$ is a collection of vector fields in $Lip^{\gamma}$ with $\gamma>p$, or a collection of linear vector fields of the form $V_{i}(Y)=A_{i}Y$, then (\ref{rde}) has a unique solution on $[0,T]$. Moreover, the Jacobian of the flow, $\frac{\partial Y_{t}}{\partial z}$, exists and satisfies the linear RDE
\begin{align}\label{linear}
d\frac{\partial Y_{t}}{\partial z}=\sum^{d}_{i=1}DV_{i}(Y_{t})\frac{\partial Y_{t}}{\partial z}d\textbf{X}^{i}_{t},\quad \frac{\partial Y_{0}}{\partial z}=I_{m}.
\end{align}
\end{tm}
\begin{proof}
For the case of $V\in Lip^{\gamma}$, the existence and uniqueness of the solution follow from Theorem 10.14 and Theorem 10.26 in \cite{Friz}. From
Theorem 10.26 in \cite{Friz} and Propsition 1 in \cite{Density}, we know that $\frac{\partial Y_{t}}{\partial z}$ exists and satisfies (\ref{linear}) with a bound
\begin{align*}
\sup_{t\in[0,T]}\left|\frac{\partial Y_{t}}{\partial z}\right|\leq |I_{m}| + C\nu_{1} \|\textbf{X}\|_{p\text{-}var;[0,T]}\exp\{C\nu_{1}^{p}\|\textbf{X}\|^{p}_{p\text{-}var;[0,T]}\},
\end{align*}
where $\nu_{1}\geq|V|_{Lip^{\gamma}}$ and $C$ depends only on $p$.

In the linear case, for any fixed initial value $z_{0}$, the existence and uniqueness of the solution follow from Theorem 10.53 in \cite{Friz}. It implies that $Y$ will not blow up on $[0,T]$ and
\begin{align}\label{Y}
\sup_{t\in[0,T]}|Y_{t}|\leq \theta + C(1+\theta)\nu_{2} \|\textbf{X}\|_{p\text{-}var;[0,T]}\exp\{c\nu_{2}^{p}\|\textbf{X}\|^{p}_{p\text{-}var;[0,T]}\}=:R(\theta),
\end{align}
where $\nu_{2}\geq\max\{|A_{i}|:i=0,...,d\}$, $\theta\geq|z_{0}|$ and $C$ depends only on $p$. This guarantees that we can localize the problem as Theorem 10.21 in \cite{Friz}.
Considering the set $\{z:|z|\leq 2\theta \}$, which is a neighborhood of $z_{0}$, and replacing $V$ by a compactly supported $\tilde{V}\in Lip^{\gamma}$ which coincides with $V$ in the ball $B_{2R}:=\{y:|y|\leq 2R(\theta)\}$, we turn the problem into the first case and obtain that $\frac{\partial Y_{t}}{\partial z}\mid_{z=z_{0}}$ exists and satisfies (\ref{linear}).
\end{proof}

\section{Symplectic structure of rough Hamiltonian system}\label{sec3}

In this section, we derive autonomous Hamiltonian equations with additional nonconservative forces via a modified Hamilton's principle. Under Assumption \ref{R}, we show that these equations form  a rough Hamiltonian system when the nonconservative forces are characterized by rough signals. Further, this rough Hamiltonian system possesses the symplectic structure similar to the deterministic case.

The classical Hamilton's principle indicates that the motion $Q(t)$ extremizes the action functional $S(Q):=\int_0^TL(Q,\dot Q)dt$ with $\delta Q(0)=\delta Q(T)=0$ under variation, where $L$ is the Lagrangian of a deterministic Hamiltonian system.
The author in \cite{Gold} gives the modified Hamilton's principle
\begin{align*}
\delta\int_0^TP\dot Q-H(P,Q)dt=0
\end{align*}
based on the Legendre transform $H(P,Q)=P\dot Q-L(Q,\dot Q)$.
However, for a Hamiltonian system influenced by additional nonconservative forces, its Hamiltonian energy turns to be
\begin{align*}
H_0(P,Q)+\sum^{d}_{i=1}H_{i}(P,Q)\dot\chi^{i}
\end{align*}
instead of $H(P,Q)$. The second term is the total work done by the additional forces and $\dot{\chi}^{i}=\frac{dX^{i}_{t}}{dt}$ represents a formal time derivative of $X^{i}$. In this case, the modified Hamilton's principle reads
\begin{align*}\label{deltaS}
\delta{S}&=\delta\int_0^TP\dot Q-H_0(P,Q)-\sum^{d}_{i=1}H_{i}(P,Q)\dot\chi^{i} dt\\
&=\int_0^T P\delta\dot Q+\dot Q\delta P-\frac{\delta H_0}{\delta P}\delta P-\frac{\delta H_0}{\delta Q}\delta Q-\sum^{d}_{i=1}\left(\frac{\delta H_i}{\delta P}\dot\chi^{i}\delta P+\frac{\delta H_i}{\delta Q}\dot\chi^{i}\delta Q\right) dt\equiv0.
\end{align*}
Integration by parts eventually yields the Hamiltonian equations of the motion with additional forces:
\begin{equation*}
\dot P=-\frac{\partial H_0(P,Q)}{\partial Q}-\sum^{d}_{i=1}\frac{\partial H_i(P,Q)}{\partial Q}\dot\chi^{i},\
\dot Q=\frac{\partial H_0(P,Q)}{\partial P}+\sum^{d}_{i=1}\frac{\partial H_i(P,Q)}{\partial P}\dot\chi^{i},
\end{equation*}
which can be rewritten as
\begin{equation}\label{symeq}
\begin{split}
dP=&-\frac{\partial H_0(P,Q)}{\partial Q}dt-\sum^{d}_{i=1}\frac{\partial H_i(P,Q)}{\partial Q}d X^{i},\quad P_{0}=p,\\
dQ=&\frac{\partial H_0(P,Q)}{\partial P}dt+\sum^{d}_{i=1}\frac{\partial H_i(P,Q)}{\partial P}d X^{i},\quad Q_{0}=q.
\end{split}
\end{equation}
Suppose $P=(P^{1},...,P^{m})^\top,Q=(Q^{1},...,Q^{m})^\top$, $p=(p^{1},...,p^{m})^\top,q=(q^{1},...,q^{m})^\top$ $\in \R^{m}$. In order to get a compact form of (\ref{symeq}), we denote $Y=(P^\top,Q^\top)^\top $, $z=(p^\top,q^\top)^\top$ $\in \R^{2m}$, $X_{t}^{0}=t$, $X_{t}=(X^{0}_{t},X^{1}_{t},X^{2}_{t},...,X^{d}_{t})^\top \in \R ^{d+1}$ and
$V_{i}=(V_{i}^{1},...,V_{i}^{2m})=(-\frac{\partial H_{i}}{\partial Q^{1}},...,-\frac{\partial H_{i}}{\partial Q^{m}},\frac{\partial H_{i}}{\partial P^{1}},...,\frac{\partial H_{i}}{\partial P^{m}})$,
$i=0,...,d$. Then (\ref{symeq}) is equivalent to
\begin{align}\label{sdeH}
dY_{t}=\sum^{d}_{i=0}V_{i}(Y_{t})dX^{i}_{t}=V(Y_{t})dX_{t},\quad  Y_{0}=z.
\end{align}

In the sequel, we consider $X^{i}_{t}=X^{i}_{t}(\omega)$, $i=1,...,d$, as Gaussian noises under the following assumption.
\begin{ap}\label{R}
Let $X^{i}_{t}$, $i=1,...,d$, be independent centered Gaussian processes with continuous sample path on $[0,T]$.
There exist some $\rho \in [1,2)$ and $K \in (0,+ \infty)$ such that for any $0\leq s<t\leq T$, the covariance of $X$ satisfies
\begin{align*}
\sup_{(t_{k}),(t_{l}) \in \mathcal{D}([s,t])}\left(\sum_{t_{i},t_{l}}|\E X_{t_{k},t_{k+1}} X_{t_{l},t_{l+1}}|^{\rho}\right)^{1/ \rho} \leq K |t-s|^{1/ \rho},
\end{align*}
where $X_{t_{k},t_{k+1}}=X_{t_{k+1}}-X_{t_{k}}$.
\end{ap}

For any $p > 2\rho$, by piecewise linear approximations, $X_{t}$ can be naturally lifted to a H$\rm \ddot{o}$lder-type weak geometric $p$-rough path $\textbf{X}_{t}$ almost surely, which takes values in $G^{[p]}(\R^{d+1})$ and $\pi_{1}(\textbf{X}_{s,t})=X_{s,t}$ (see e.g., \cite[Theorem 15.33]{Friz}). Therefore, (\ref{symeq}) or (\ref{sdeH}) can be transformed into an RDE almost surely and interpreted in the framework of rough path theory. In this sense, we call (\ref{symeq}) or (\ref{sdeH}) a rough Hamiltonian system.

The fBm with Hurst parameter $H \in (1/4,1/2]$ satisfies Assumption \ref{R} with $\rho=\frac{1}{2H}$, which is frequently used in practical applications. When $X_{t}^{i}$, $i=1,...,d$, are Brownian motions, i.e., $H=1/2$, $Y_{t}$ equals to the solution of a stochastic Hamiltonian system driven by standard Brownian motions (see \cite{additive,Milstein}) in Stratonovich sense almost surely.

As a function of time $t$ and initial value $z$, $Y$ is a phase flow for almost every $\omega$. In the deterministic case and the stochastic case, we know that the phase flow preserves the symplectic structure, i.e., the differential $2$-form $dP \wedge dQ$ is invariant. Note that the differential here is made with respect to the initial value, which is different from the formal time derivative in (\ref{symeq}). The geometric interpretation is that the sum of the oriented areas of a two-dimensional surface, obtained by projecting the phase flow onto the coordinate planes ($p^{1},q^{1}$),...,($p^{m},q^{m}$), is an integral invariant. The rough Hamiltonian systems, a generalization for the deterministic case and the stochastic case but allowing rougher noises, share this characteristic property of Hamiltonian systems, which is proved in the next Theorem.


\begin{tm}\label{symp}
The phase flow of the rough Hamiltonian system (\ref{symeq}) preserves the symplectic structure:
\begin{align*}
dP \wedge dQ=dp\wedge dq,\quad a.s.
\end{align*}
\end{tm}
\begin{proof}
From Theorem \ref{well}, we know that $P$ and $Q$ are differentiable with respect to $p$ and $q$.
Denote $P^{jk}_{p}=\frac{\partial P^{j}}{\partial p^{k}}$, $Q^{jk}_{p}=\frac{\partial Q^{j}}{\partial p^{k}}$,
$P^{jk}_{q}=\frac{\partial P^{j}}{\partial q^{k}}$, $Q^{jk}_{q}=\frac{\partial Q^{j}}{\partial q^{k}}$.
Since $dP^{j}=\sum\limits^{m}_{k=1}P^{jk}_{p}dp^{k}+\sum\limits^{m}_{l=1}P^{jl}_{q}dq^{l}$ and $dQ^{j}=\sum\limits^{m}_{k=1}Q^{jk}_{p}dp^{k}+\sum\limits^{m}_{l=1}Q^{jl}_{q}dq^{l}$, we have
\begin{align*}
dP \wedge dQ=&\sum^{m}_{j=1}dP^{j} \wedge dQ^{j}\\
=&\sum^{m}_{j=1}\sum^{m}_{k=1}\sum^{m}_{l=1}\left(P^{jk}_{p}Q^{jl}_{q}-P^{jl}_{q}Q^{jk}_{p}\right)dp^{k}\wedge dq^{l}\\
&+\sum^{m}_{j=1}\sum^{m}_{k=1}\sum^{k-1}_{l=1}\left(P^{jk}_{p}Q^{jl}_{p}-P^{jl}_{p}Q^{jk}_{p}\right)dp^{k}\wedge dq^{l}\\
&+\sum^{m}_{j=1}\sum^{m}_{k=1}\sum^{k-1}_{l=1}\left(P^{jk}_{q}Q^{jl}_{q}-P^{jl}_{q}Q^{jk}_{q}\right)dp^{k}\wedge dq^{l}.
\end{align*}
Meanwhile, the time derivatives of $P^{jk}_{p}$,  $Q^{jk}_{p}$, $P^{jk}_{q}$, $Q^{jk}_{q}$ yield
\begin{equation*}
\begin{split}
&dP^{jk}_{p}=\sum_{i=0}^{d}\sum^{m}_{r=1}\left(
-\frac{\partial^{2} H_{i}}{\partial Q^{j}\partial P^{r}} P^{rk}_{p}
-\frac{\partial^{2} H_{i}}{\partial Q^{j}\partial Q^{r}} Q^{rk}_{p}
 \right)dX^{i},\ P^{jk}_{p}(t_{0})=\delta_{jk},\\
 &dQ^{jk}_{p}=\sum_{i=0}^{d}\sum^{m}_{r=1}\left(
\frac{\partial^{2} H_{i}}{\partial P^{j}\partial P^{r}} P^{rk}_{p}
+\frac{\partial^{2} H_{i}}{\partial P^{j}\partial Q^{r}} Q^{rk}_{p}
 \right)dX^{i},\ Q^{jk}_{p}(t_{0})=0,\\
 &dP^{jk}_{q}=\sum_{i=0}^{d}\sum^{m}_{r=1}\left(
-\frac{\partial^{2} H_{i}}{\partial Q^{j}\partial P^{r}} P^{rk}_{q}
-\frac{\partial^{2} H_{i}}{\partial Q^{j}\partial Q^{r}} Q^{rk}_{q}
 \right)dX^{i},\ P^{jk}_{q}(t_{0})=0,\\
 &dQ^{jk}_{q}=\sum_{i=0}^{d}\sum^{m}_{r=1}\left(
\frac{\partial^{2} H_{i}}{\partial P^{j}\partial P^{r}} P^{rk}_{q}
+\frac{\partial^{2} H_{i}}{\partial P^{j}\partial Q^{r}} Q^{rk}_{q}
 \right)dX^{i},\ Q^{jk}_{q}(t_{0})=\delta_{jk},\\
\end{split}
\end{equation*}
where all coefficients are calculated at $(P,Q)$.
Then one can check
\begin{equation*}
\begin{split}
&d \left(\sum^{m}_{j=1} \left( P^{jk}_{p}Q^{jl}_{q}- P^{jl}_{q}Q^{jk}_{p} \right) \right)=0,\ \forall\ k,l,\\
&d \left(\sum^{m}_{j=1} \left( P^{jk}_{p}Q^{jl}_{p}- P^{jl}_{p}Q^{jk}_{p} \right) \right)=0,\ \forall\ k\neq l,\\
&d \left(\sum^{m}_{j=1} \left( P^{jk}_{q}Q^{jl}_{q}- P^{jl}_{q}Q^{jk}_{q} \right) \right)=0,\ \forall\ k\neq l,
\end{split}
\end{equation*}
to get $\sum\limits^{m}_{j=1}dP^{j} \wedge dQ^{j}=\sum\limits^{m}_{j=1}dp^{j} \wedge dq^{j}$.
\end{proof}

Because of the symplecticity of the phase flow of a rough Hamiltonian system, it is natural to construct numerical methods to inherit this property. However, the simplified step-$N$ Euler schemes cannot inherit this property. Consequently, in the next three sections we will propose and analyze Runge--Kutta methods for rough Hamiltonian systems.

\section{Runge--Kutta methods}\label{sec4}
Given a time step $h$, we construct $s$-stage Runge--Kutta method by
\begin{equation}\label{RK}
\begin{split}
Y^{h}_{k}(\alpha)&=Y^{h}_{k}+\sum^{s}_{\beta=1}a_{\alpha\beta}V(Y^{h}_{k}(\beta))X_{t_{k},t_{k+1}},  \\
Y^{h}_{k+1}&=Y^{h}_{k}+\sum^{s}_{\alpha=1}b_{\alpha}V(Y^{h}_{k}(\alpha))X_{t_{k},t_{k+1}}
\end{split}
\end{equation}
with coefficients $a_{\alpha\beta}$, $b_{\alpha}$, $\alpha,\beta=1,...,s$, $t_{k}=kh$, $k=1,...,N$ and $Y^{h}_{0}=z$. Here, we suppose $N=T/h \in \N$ for simplicity. The numerical solution $Y^{h}_{k}$ if exists is an approximation for $Y_{t_{k}}$.

\begin{tm}\label{symcon}
The $s$-stage Runge--Kutta method (\ref{RK}) inherits the symplectic structure of a rough Hamiltonian system, if the coefficients satisfy
$$a_{\alpha\beta}b_{\alpha}+a_{\beta\alpha}b_{\beta}=b_{\alpha}b_{\beta},\ \forall\ \alpha,\beta=1,...,s.$$
\end{tm}
\begin{proof}
Denoting the $j$th-component of $(P^{h}_{k},Q^{h}_{k})$ by $(P^{h,j}_{k},Q^{h,j}_{k})$, $j=1,...,m$, then the exterior derivatives of $P^{h,j}_{k+1}$ and $Q^{h,j}_{k+1}$ show that
\begin{equation*}
\begin{split}
dP^{h,j}_{k+1}&=dP^{h,j}_{k}-\sum^{s}_{\alpha=1}\sum^{d}_{i=0}\sum^{m}_{r=1}b_{\alpha}\left(\frac{\partial^{2} H_{i}}{\partial Q^{j}\partial P^{r}}dP^{h,r}_{k}(\alpha)+\frac{\partial^{2} H_{i}}{\partial Q^{j}\partial Q^{r}}dQ^{h,r}_{k}(\alpha)\right)X^{i}_{t_{k},t_{k+1}},\\
dQ^{h,j}_{k+1}&=dQ^{h,j}_{k}+\sum^{s}_{\alpha=1}\sum^{d}_{i=0}\sum^{m}_{r=1}b_{\alpha}\left(\frac{\partial^{2} H_{i}}{\partial P^{j}\partial P^{r}}dP^{h,r}_{k}(\alpha)+\frac{\partial^{2} H_{i}}{\partial P^{j}\partial Q^{r}}dQ^{h,r}_{k}(\alpha)\right)X^{i}_{t_{k},t_{k+1}}.
\end{split}
\end{equation*}
Exterior product performed between above equations yield
\begin{equation*}
\begin{split}
&dP^{h,j}_{k+1}\wedge dQ^{h,j}_{k+1}=dP^{h,j}_{k}\wedge dQ^{h,j}_{k}\\
&+\sum^{d}_{i=0}\sum^{m}_{r=1}\sum^{s}_{\alpha=1}b_{\alpha}dP^{h,j}_{k}\wedge\left(\frac{\partial^{2} H_{i}}{\partial P^{j}\partial P^{r}} dP^{h,r}_{k}(\alpha)+\frac{\partial^{2} H_{i}}{\partial P^{j}\partial Q^{r}} dQ^{h,r}_{k}(\alpha)\right)X^{i}_{t_{k},t_{k+1}}\\
&-\sum^{d}_{i=0}\sum^{m}_{r=1}\sum^{s}_{\alpha=1}b_{\alpha}\left(\frac{\partial^{2} H_{i}}{\partial Q^{j}\partial P^{r}}dP^{h,r}_{k}(\alpha)+\frac{\partial^{2} H_{i}}{\partial Q^{j}\partial Q^{r}}dQ^{h,r}_{k}(\alpha)\right)\wedge dQ^{h,j}_{k}X^{i}_{t_{k},t_{k+1}}\\
&-\left(\sum^{d}_{i_{1}=0}\sum^{m}_{r_{1}=1}\sum^{s}_{\alpha=1}b_{\alpha}\left(\frac{\partial^{2} H_{i_{1}}}{\partial Q^{j}\partial P^{r_{1}}}dP^{h,r_{1}}_{k}(\alpha)+\frac{\partial^{2} H_{i_{1}}}{\partial Q^{j}\partial Q^{r_{1}}}dQ^{h,r_{1}}_{k}(\alpha)\right)X^{i_{1}}_{t_{k},t_{k+1}}\right)\\
&\wedge \left( \sum^{d}_{i_{2}=0}\sum^{m}_{r_{2}=1}\sum^{s}_{\alpha=1}b_{\alpha}\left(\frac{\partial^{2} H_{i_{2}}}{\partial P^{j}\partial P^{r_{2}}} dP^{h,r_{2}}_{k}(\alpha)+\frac{\partial^{2} H_{i_{2}}}{\partial P^{j}\partial Q^{r_{2}}}dQ^{h,r_{2}}_{k}(\alpha)\right)X^{i_{2}}_{t_{k},t_{k+1}}  \right).
\end{split}
\end{equation*}
Replacing $dP^{h,j}_{k}$ and $dQ^{h,j}_{k}$ in the above equation by the following expressions
\begin{equation*}
\begin{split}
dP^{h,j}_{k}&=dP^{h,j}_{k}(\alpha)+\sum^{s}_{\beta=1}\sum^{d}_{i=0}\sum^{m}_{r=1}a_{\alpha\beta}\left(\frac{\partial^{2} H_{i}}{\partial Q^{j}\partial P^{r}}dP^{h,r}_{k}(\beta)+\frac{\partial^{2} H_{i}}{\partial Q^{j}\partial Q^{r}}dQ^{h,r}_{k}(\beta)\right)X^{i}_{t_{k},t_{k+1}},\\
dQ^{h,j}_{k}&=dQ^{h,j}_{k}(\alpha)-\sum^{s}_{\beta=1}\sum^{d}_{i=0}\sum^{m}_{r=1}a_{\alpha\beta}\left(\frac{\partial^{2} H_{i}}{\partial P^{j}\partial P^{r}}dP^{h,r}_{k}(\beta)+\frac{\partial^{2} H_{i}}{\partial P^{j}\partial Q^{r}}dQ^{h,r}_{k}(\beta)\right)X^{i}_{t_{k},t_{k+1}},
\end{split}
\end{equation*}
we obtain the symplectic condition
$$a_{\alpha\beta}b_{\alpha}+a_{\beta\alpha}b_{\beta}=b_{\alpha}b_{\beta},\ \forall\ \alpha,\beta=1,...,s.$$
This ensures the discrete symplectic structure $dP^{h}_{k+1}\wedge dQ^{h}_{k+1}=dP^{h}_{k}\wedge dQ^{h}_{k}$.
\end{proof}

Indeed, symplectic Runge--Kutta methods are implicit in general. As we said before, the solvability of symplectic Runge--Kutta methods should be given. Based on the Brouwer's theorem, we consider the bounded case and the linear case for the vector field $V$ respectively.
\begin{prop}\label{sol}
If $V=(V_{i})_{0\leq i\leq d}$ is a collection of vector fields in $Lip^{\gamma}$ for some $\gamma>0$, then for arbitrary time step $h>0$, initial value and coefficients $\{a_{\alpha\beta}:\alpha,\beta=1,...,s\}$, the $s$-stage Runge--Kutta method (\ref{RK}) has at least one solution for any $\omega$.
\end{prop}
\begin{proof}
Fix $h$ and $Y_{k}^{h}$. Denote $Z_{1},...,Z_{s}\in\R^{2m}$ and $Z=(Z_{1}^\top,...,Z_{s}^\top)^\top\in\R^{2ms}$. We define a map $\phi:\R^{2ms}\rightarrow \R^{2ms}$ with
\begin{align*}
\begin{split}
\phi(Z)&=(\phi(Z)_{1}^\top,...,\phi(Z)_{s}^\top)^\top,\\ \phi(Z)_{\alpha}&=Z_{\alpha}-Y_{k}^{h}-\sum^{s}_{\beta=1}a_{\alpha\beta}V(Z_{\beta})X_{t_{k},t_{k+1}}(\omega),\ \alpha=1,...,s.
\end{split}
\end{align*}
Let $c\geq\max\{|a_{\alpha\beta}|:\alpha,\beta=1,...,s\}$,  $\nu\geq|V|_{Lip^{\gamma}}$ and $R=\sqrt{s}|Y_{k}^{h}|+s^{2}c\nu|X_{t_{k},t_{k+1}}(\omega)|+1$, we have that for any $|Z|= R$,
\begin{align*}
Z^\top\phi(Z)&=\sum^{s}_{\alpha=1}Z_{\alpha}^\top\left(Z_{\alpha}-Y_{k}^{h}-\sum^{s}_{\beta=1}a_{\alpha\beta}V(Z_{\beta})X_{t_{k},t_{k+1}}(\omega)\right)\\
&\geq|Z|\left(|Z|-\sqrt{s}|Y_{k}^{h}|-s^{2}c\nu|X_{t_{k},t_{k+1}}(\omega)|\right)>0.
\end{align*}

We aim to show that $\phi(Z)=0$ has a solution in the ball $B_{R}:=\{Z:|Z|\leq R\}$.
Assume by contradiction that $\phi(Z)\neq 0$ for any $|Z|\leq R$. We define a continuous map $\psi$ by $\psi(Z)=-\frac{R\phi(Z)}{|\phi(Z)|}$. Since $\psi: B_{R}\rightarrow B_{R}$, $\psi$ has at least one fixed point $Z^{*}$ such that $|Z^{*}|=|\psi(Z^{*})|=R$. This leads to a contradiction that $|Z^{*}|^{2}=\psi(Z^{*})^\top Z^{*}=-\frac{RZ^\top\phi(Z)}{|\phi(Z)|}<0$. Therefore, $\phi$ has at least one solution and then (\ref{RK}) is valid.
\end{proof}

Things in the linear case are more complicated, then additional conditions should be considered. We focus on the $1$-stage symplectic Runge--Kutta method ($a_{11}=\frac{1}{2},b_{1}=1$), i.e., the midpoint scheme:
\begin{align}\label{midpoint}
\begin{split}
Y^{h}_{k+1/2}&=Y^{h}_{k}+\frac{1}{2}V(Y^{h}_{k+1/2})X_{t_{k},t_{k+1}},\\
Y^{h}_{k+1}&=Y^{h}_{k}+V(Y^{h}_{k+1/2})X_{t_{k},t_{k+1}}.
\end{split}
\end{align}

\begin{prop}\label{linearcase}
Suppose $V=(V_{i})_{0\leq i\leq d}$ is a collection of linear vector fields of the form $V_{i}(Y)=A_{i}Y$. If $A_{i}$, $i=0,...,d$ are all skew symmetric, i.e., $A_{i}=-A_{i}^\top$, then for arbitrary time step $h>0$ and initial value, the the midpoint scheme is solvable  for any $\omega$.
\end{prop}
\begin{proof}
Fix $h$ and $Y_{k}^{h}$. We define a map $\phi:\R^{2m}\rightarrow \R^{2m}$ with
$$\phi(Z)=Z-Y_{k}^{h}-\frac{1}{2}V(Z)X_{t_{k},t_{k+1}}(\omega).$$
Since $A_{i}$, $i=0,...,d$, are all skew symmetric, we have $Z^\top A_{i}Z=0$, which leads to
\begin{align*}
Z^\top\phi(Z)&=|Z|^{2}-Z^\top Y_{k}^{h}-\frac{1}{2}\sum^{d}_{i=0}Z^\top A_{i}ZX_{t_{k},t_{k+1}}^{i}(\omega)\geq|Z|(|Z|-|Y_{k}^{h}|).
\end{align*}
From Proposition \ref{sol}, we know that $\phi$ has at least a solution in the ball $\{Z:|Z|\leq |Y_{k}^{h}|+1\}$.
This implies that the midpoint scheme is solvable.
\end{proof}

\begin{rk}
For a linear rough Hamiltonian system, i.e., (\ref{sdeH}) can be rewritten equivalently into the form (\ref{symeq}), if in addition $A_{i}\in\R^{2m\times 2m}$ is skew symmetric, then $A_{i}$ is in the form of $\left( \begin{array}{cc}
A_{i}^{1} & -A_{i}^{2} \\
A_{i}^{2} & A_{i}^{1}\end{array} \right)$, where $A_{i}^{1}=-(A_{i}^{1})^\top \in \R^{m\times m}$ and $A_{i}^{2}=(A_{i}^{2})^\top \in \R^{m\times m}$.
\end{rk}

\section{Convergence analysis}\label{sec5}

In Definition \ref{solution}, the solution of an RDE depends on the information about iterated integrals of the noise, which is very difficult if possible to be simulated. The Runge--Kutta methods are implementable because they only make use of the increments of the noise and omit the antisymmetric part of iterated integrals, which is called L$\acute{e}$vy's area in the Brownian motion setting. To analyze the error caused by this,
we introduce piecewise linear approximations, which is a special case of Wong--Zakai approximations.
We define $x^{h}_{t}:=(x^{h,0}_{t},...,x^{h,d}_{t})$ with
\begin{align}\label{xh}
x^{h,i}_{t}:= X^{i}_{t_{k}}+\dfrac{t-t_{k}}{h}X^{i}_{t_{k},t_{k+1}},\ t \in (t_{k},t_{k+1}],\ \forall\ k=1,...,N,\ i=0,...,d,
\end{align}
and consider the following differential equation in Riemann-Stieltjes integral sense:
\begin{align}\label{ode}
dy^{h}_{t}=V(y^{h}_{t})dx^{h}_{t},\quad y^{h}_{0}=z.
\end{align}

\begin{tm}\label{WZk}
Suppose $X$ and $x^{h}$ are as in Assumption \ref{R} and (\ref{xh}) respectively. If $V=(V_{i})_{0\leq i\leq d}$ is a collection of vector fields in $Lip^{\gamma}$ for some $\gamma>2\rho$, or a collection of linear vector fields of the form $V_{i}(Y)=A_{i}Y$, then both (\ref{sdeH}) and (\ref{ode}) have unique solutions $Y$ and $Y^{h}$ almost surely.

Further, let $\theta\geq|z|$, $\nu\geq\max\{|V|_{Lip^{\gamma}},|A_{i}|,i=0,,,,d\}$, then for any $0\leq\eta < \min\{\frac{1}{\rho} -\frac{1}{2}, \frac{1}{2\rho} - \frac{1}{\gamma}\}$, there exits a finite random variable $C(\omega)$ and a null set $M$ such that
\begin{align*}
\sup_{t \in [0,T]}|y_{t}^{h}(\omega)-Y_{t}(\omega)| \leq C(\omega)h^{\eta},\quad \forall\ \omega\in \Omega \setminus M
\end{align*}
holds for any $h\geq0$. $C(\omega)$ depends on $\rho$, $\eta$, $\gamma$, $\nu$, $\theta$, $K$ and $T$.
\end{tm}
\begin{proof}
For the first case of $V$, see Theorem 6 and Corollary 8 in \cite{Wzk}. $C(\omega)$ depends on $\rho$, $\eta$, $\gamma$, $\nu$, $K$ and $T$.

For the linear case, we know if $p>2\rho$, then for any $\omega\in\Omega\setminus M$, $\|S_{[p]}(x^{h})(\omega)\|_{p\text{-}var;[0,T]}$ has an upper bound for all $h\geq0$ (see \cite[Theorem 15.33]{Friz}). According to the estimate in (\ref{Y}), there exists some constant $R>0$ depending on $\omega$, $\rho$, $\theta$, $\eta$, $\nu$ and $T$, such that it is a bound for $Y_{t}$ and $y^{h}_{t}$, for all $h\geq0$ on $[0,T]$. Using again localization, we replace $V$ by a compactly supported $\tilde{V}\in Lip^{\gamma}$ which coincides with $V$ in the ball $B_{R}:=\{y:|y|\leq R\}$ and then the result can be generalized from the first case to this linear case.
\end{proof}

We are in position to analyze the convergence rate. If $V\in Lip^{\gamma}$, we analyze the midpoint scheme for $\rho\in[1,3/2)$ in Theorem \ref{midrate} and two $s$-stage Runge--Kutta methods ($s>1$) for $\rho\in[1,2)$ in Corollary \ref{RKrate}. If $V$ is linear, we focus on the convergence rate of the midpoint scheme for $\rho\in[1,3/2)$ in Theorem \ref{midlinear}, using the uniform boundedness of numerical solutions in Proposition \ref{solb}.

\begin{tm}\label{midrate}
Suppose $\rho \in [1,3/2)$ and $Y_{k}^{h}$ is the numerical solution to the midpoint scheme (\ref{midpoint}). If $|V|_{Lip^{\gamma}}\leq\nu<\infty$ with $\gamma>2\rho$, then for any $0<\eta < \min\{\frac{3}{2\rho}-1, \frac{1}{2\rho} - \frac{1}{\gamma}\}$, there exists a finite random variable $C(\omega)$ and a null set $M$ such that
\begin{align*}
\max_{k=1,...,N}|Y^{h}_{k}(\omega)-Y_{t_{k}}(\omega)|\leq C(\omega)h^{\eta},\quad  \forall\ \omega \in \Omega\setminus M
\end{align*}
holds for any $0<h\leq 1$. $C(\omega)$ depends on $\rho$, $\eta$, $\gamma$, $\nu$, $K$ and $T$.
\end{tm}
\begin{proof}
We divide the error into two parts
\begin{align}\label{twoparts}
\max_{k=1,...,N}|Y^{h}_{k}-Y_{t_{k}}|\leq \max_{k=1,...,N}|Y^{h}_{k}-y^{h}_{t_{k}}|+\sup_{t\in[0,T]}|y^{h}_{t}-Y_{t}|.
\end{align}
The second part has been estimated in Theorem \ref{WZk}, then it suffices to estimate the first part.

$\it{Step\ 1.\ Local\ error.}$ According to the definition of $x^{h}_{t}$, for $k=1,...,N$, $n\geq1$, $i_{1},...,i_{n}\in\{0,...,d\}$, we have
\begin{align*}
\int_{t_{k-1}\leq u_{1}<\cdot\cdot\cdot<u_{n}\leq t_{k}}dx^{h,i_{1}}_{u_{1}}\cdot\cdot\cdot dx^{h,i_{n}}_{u_{n}}=&\int_{t_{k-1}\leq u_{1}<\cdot\cdot\cdot<u_{n}\leq t_{k}}\frac{X^{i_{1}}_{t_{k-1},t_{k}}}{h}du_{1}\cdot\cdot\cdot\frac{X^{i_{n}}_{t_{k-1},t_{k}}}{h}du_{n}\\
=&\frac{X^{i_{1}}_{t_{k-1},t_{k}}\cdot\cdot\cdot X^{i_{n}}_{t_{k-1},t_{k}}}{n!},
\end{align*}
which concludes $\int_{t_{k-1}\leq u_{1}<\cdot\cdot\cdot<u_{n}\leq t_{k}}dx^{h}_{u_{1}}\otimes\cdot\cdot\cdot \otimes dx^{h}_{u_{n}}=X_{t_{k-1},t_{k}}^{\otimes n}/n!$. Taylor expansion yields
\begin{align*}
y^{h}_{t_{1}}&=y^{h}_{0}+\int_{0\leq u_{1}\leq t_{1}}V(y^{h}_{u_{1}})dx^{h}_{u_{1}}\\
&=y^{h}_{0}+V(y^{h}_{0})X_{0,t_{1}}+\int_{0\leq u_{2}<u_{1}\leq t_{1}}DV(y^{h}_{u_{2}})V(y^{h}_{u_{2}})dx^{h}_{u_{2}}\otimes dx^{h}_{u_{1}}\\
&=y^{h}_{0}+V(y^{h}_{0})X_{0,t_{1}}+DV(y^{h}_{0})V(y^{h}_{0})X_{0,t_{1}}^{\otimes 2}/2+R_{1}
\end{align*}
with
\begin{align*}
R_{1}=\int_{0\leq u_{3}<u_{2}<u_{1}\leq t_{1}}\left(D^{2}V(y^{h}_{u_{3}})V(y^{h}_{u_{3}})^{2}+DV(y^{h}_{u_{3}})^{2}V(y^{h}_{u_{3}})\right)dx^{h}_{u_{3}}\otimes dx^{h}_{u_{2}}\otimes dx^{h}_{u_{1}}.
\end{align*}
As for $V(Y^{h}_{1/2})$ in (\ref{midpoint}), we have
\begin{align*}
V\left(Y^{h}_{1/2}\right)=&V\left(Y^{h}_{0}\right)+DV(Y^{h}_{0})(Y^{h}_{1}-Y^{h}_{0})/2+D^{2}V(\tilde{Y}^{h}_{0})(Y^{h}_{1}-Y^{h}_{0})^{\otimes2}/8,
\end{align*}
where $\tilde{Y}^{h}_{0}$ is between $Y^{h}_{0}$ and $Y^{h}_{1/2}$. Then,
\begin{align*}
Y^{h}_{1}=Y^{h}_{0}+V\left(Y^{h}_{0}\right)X_{0,t_{1}}+DV(Y^{h}_{0})V(Y^{h}_{0})X_{0,t_{1}}^{\otimes2}/2+R_{2}
\end{align*}
with
\begin{align*}
R_{2}=&\left(DV(Y^{h}_{0})\right)^{2}(Y^{h}_{1}-Y^{h}_{0})X_{0,t_{1}}^{\otimes2}/4+D^{2}V(\tilde{Y}^{h}_{0})(Y^{h}_{1}-Y^{h}_{0})^{\otimes 2}X_{0,t_{1}}/8\\
&+DV(Y^{h}_{0})D^{2}V(\tilde{Y}^{h}_{0})(Y^{h}_{1}-Y^{h}_{0})^{\otimes 2}X_{0,t_{1}}^{\otimes2}/16.
\end{align*}
Noticing $|Y^{h}_{1}-Y^{h}_{0}|\leq|V(Y^{h}_{1/2})\|X_{0,t_{1}}|$ and $\gamma>2$, we obtain the difference between $Y^{h}_{1}$ and $y^{h}_{t_{1}}$ is
\begin{align*}
|Y^{h}_{1}-y^{h}_{t_{1}}|\leq |R_{2}-R_{1}| \leq \max\{|V|_{Lip^{\gamma}}^{4},1\}(|X_{0,t_{1}}|^{3}+|X_{0,t_{1}}|^{4}).
\end{align*}
For any $0<\eta < \frac{3}{2\rho}-1$, there exists a $p>2\rho$, such that $\eta=\frac{3}{p}-1$. Since for any $\omega\in\Omega\setminus M$, $|X_{0,t_{1}}(\omega)| =|\pi_{1}(\textbf{X}_{0,t_{1}}(\omega))|\leq \|\textbf{X}(\omega)\|_{1/p\text{-}{\rm H\ddot{o}l};[0,T]}h^{1/p}$, the local error is
\begin{align*}
|Y^{h}_{1}(\omega)-y^{h}_{t_{1}}(\omega)|\leq C_{1}(\omega)h^{3/p},\quad \forall\ h\leq 1.
\end{align*}

$\it{Step\ 2.\ Global\ error.}$ We use the notation $\pi(t_{0},y_{0},x)_{t}$, $t \geq t_{0}$, to represent the solution of (\ref{ode}) with the driven noise $x$ and the initial value $y_{0}$ at time $t_{0}$. For any $\omega\in\Omega\setminus M$, we know that $\|S_{[p]}(x^{h})(\omega)\|^{p}_{p\text{-}var;[0,T]}$ has an upper bound for all $h>0$ (see \cite{Friz}, Theorem 15.33). Since $\gamma>2\rho$, similar to Theorem 10.30 in \cite{Friz}, there exists some constant $C=C(\gamma,p)$ such that for any  $k=1,...,N$,
\begin{align*}
|Y^{h}_{k}(\omega)-y_{t_{k}}^{h}(\omega)|\leq& \sum^{k}_{s=1}|\pi(t_{s},Y^{h}_{s},x^{h})_{t_{k}}-\pi(t_{s-1},Y^{h}_{s-1},x^{h})_{t_{k}}|\\
\leq& \sum^{k}_{s=1}\exp\{C\nu^{p}\|S_{[p]}(x^{h})\|^{p}_{p\text{-}var;[t_{s},t_{k}]}\}|Y^{h}_{s}-\pi(t_{s-1},Y^{h}_{s-1},x^{h})_{t_{s}}|\\
\leq& \sum^{k}_{s=1}\exp\{C\nu^{p}\|S_{[p]}(x^{h})(\omega)\|^{p}_{p\text{-}var;[0,T]}\}C_{1}(\omega)h^{3/p}\\
\leq& C(\omega)h^{3/p-1}.
\end{align*}
Therefore, for any $\omega\in \Omega \setminus M$,
\begin{align}\label{scheor}
\max_{k=1,...,N}|Y^{h}_{k}(\omega)-y_{t_{k}}^{h}(\omega)|\leq C(\omega)h^{\eta},\quad \forall\ \eta<3/2\rho-1.
\end{align}
The fact that $\frac{3}{2\rho} -1 \leq \frac{1}{\rho} -\frac{1}{2}$ for $\rho \in [1,3/2)$ and Theorem \ref{WZk} complete the proof.
\end{proof}

When the components of the noise are independent standard Brownian motions, i.e., $\rho=1$, the convergence order is (almost) consistent with the mean square convergence rate of SDEs in Stratonovich sense with multiplicative noise. Moreover,
note that the estimate in Theorem \ref{midrate} is valid for the case when $\rho \in [1,3/2)$, which is similar to the simplified step-2 Euler scheme in \cite{Wzk}.
To construct a numerical scheme for $\rho\in[1,2)$ and fill the gap between the convergence rates of two parts in (\ref{twoparts}), we use a higher stage symplectic Runge--Kutta method with local order $\tau\geq4$ when applied to classical ordinary differential equations. If $\gamma> \max\{2\rho,\tau-1\}$, then the estimate in (\ref{scheor}) for the global error will hold for any $0<\eta<\frac{\tau}{2\rho}-1$. Two symplectic Runge--Kutta methods with $\tau=5$ and $\tau=4$ are given in Corollary \ref{RKrate} to illustrate this point.

\begin{cor}\label{RKrate}
Suppose $\rho \in [1,2)$ and $Y_{k}^{h}$ represents the solution to the Runge--Kutta methods with coefficients expressed in the Butcher tableaus below.

Method \uppercase\expandafter{\romannumeral1}:
\begin{displaymath}
\begin{array}{c|ccc}
(3-\sqrt{3})/6&1/4&(3-2\sqrt{3})/12\\
(3+\sqrt{3})/6&(3+2\sqrt{3})/12&1/4\\
\hline
\ &1/2&1/2
\end{array}.
\end{displaymath}

Method \uppercase\expandafter{\romannumeral2}:
\begin{displaymath}
\begin{array}{c|ccc}
a/2&a/2&\ &\ \\
3a/2&a&a/2&\ \\
1/2+a&a&a&1/2-a\\
\hline
\ &a&a&1-2a
\end{array},
\end{displaymath}
where $a=1.351207$ is the real root of $6x^{3}-12x^{2}+6x-1=0$.

If $|V|_{Lip^{\gamma}}\leq\nu<\infty$ with $\gamma >4$ for Method \uppercase\expandafter{\romannumeral1} ($\gamma >\max\{2\rho,3\}$ for Method \uppercase\expandafter{\romannumeral2}),
then for any $0<\eta<\min\{\frac{1}{\rho}-\frac{1}{2},\frac{1}{2\rho}-\frac{1}{\gamma}\}$, there exists a finite random variable $C(\omega)$ and a null set $M$ such that
\begin{align*}
\max_{k=1,...,N}|Y^{h}_{k}(\omega)-Y_{t_{k}}(\omega)|\leq C(\omega)h^{\eta},\quad \forall\ \omega \in \Omega\setminus M.
\end{align*}
In addition, both of the two methods inherit the symplectic structure of a rough Hamiltonian system.
\end{cor}
\begin{proof}
According to \cite{GeometricEH}, we know that the corresponding $\eta$ in (\ref{scheor}) belongs to $(0,\frac{5}{2\rho}-1)$ for Method \uppercase\expandafter{\romannumeral1} and $(0,\frac{2}{\rho}-1)$ for Method \uppercase\expandafter{\romannumeral2}. This result together with Theorem \ref{WZk} finally leads to the global error estimate. Moreover, the coefficients satisfy the symplectic condition in Theorem \ref{symcon}.
\end{proof}

\begin{rk}
This convergence rate equals to the simplified step-3 Euler scheme for $\gamma$ large enough. Actually, using Runge--Kutta methods with larger $\tau$ will not improve the convergence rate since the error of piecewise linear approximations persists.
\end{rk}

The next proposition is essential to get the convergence rate of the midpoint scheme in the linear case.

\begin{prop}\label{solb}
If $V$ in (\ref{sdeH}) is a collection of skew symmetric linear vector fields of the form $V_{i}(Y)=A_{i}Y$, then numerical solutions from the midpoint scheme (\ref{midpoint}) are uniformly bounded. More precisely, $|Y^{h}_{k}|=|z|$, $k=1,...,N$.
\end{prop}
\begin{proof}
For any $k=0,...,N-1$, since $Y^{h}_{k}=Y^{h}_{k+1/2}-\frac{1}{2}\sum^{d}_{i=0}A_{i}Y^{h}_{k+1/2}X_{t_{k},t_{k+1}}^{i},$ we have
\begin{align*}
|Y^{h}_{k+1}|^{2}=&|Y^{h}_{k}|^{2}
+ \sum^{d}_{i=0}(Y^{h}_{k+1/2})^\top A_{i} ^\top Y^{h}_{k} X_{t_{k},t_{k+1}}^{i} + \sum^{d}_{i=0}(Y^{h}_{k})^\top A_{i}Y^{h}_{k+1/2} X_{t_{k},t_{k+1}}^{i}\\
&+\sum^{d}_{i_{1}=0}\sum^{d}_{i_{2}=0}(Y^{h}_{k+1/2})^\top A_{i_{1}} ^\top  A_{i_{2}} Y^{h}_{k+1/2} X_{t_{k},t_{k+1}}^{i_{1}}X_{t_{k},t_{k+1}}^{i_{2}}\\
=&|Y^{h}_{k}|^{2}+2\sum^{d}_{i=0}(Y^{h}_{k+1/2})^\top  A_{i}Y^{h}_{k+1/2} X_{t_{k},t_{k+1}}^{i}.
\end{align*}
Then, $|Y^{h}_{k+1}|^{2}=|Y^{h}_{k}|^{2}=|z|$ from $(Y^{h}_{k+1/2})^\top A_{i} Y^{h}_{k+1/2}=0$.
\end{proof}

\begin{tm}\label{midlinear}
Suppose $\rho \in [1,3/2)$. Let $V$ be as in Proposition \ref{solb} and $Y_{k}^{h}$ represents the numerical solution to the midpoint scheme. If $\theta\geq|z|$, $\nu\geq\max\{|A_{i}|:i=0,...,d\}$, then for any $\gamma>2\rho$, $0<\eta < \min\{\frac{3}{2\rho}-1,\frac{1}{2\rho}-\frac{1}{\gamma}\}$, there exists a finite random variable $C(\omega)$ and a null set $M$ such that
\begin{align*}
\max_{k=1,...,N}|Y^{h}_{k}(\omega)-Y_{t_{k}}(\omega)|\leq C(\omega)h^{\eta},\quad \forall\ \omega \in \Omega\setminus M
\end{align*}
holds for any $0<h\leq 1$. $C(\omega)$ depends on $\rho$, $\eta$, $\gamma$, $\nu$, $\theta$, $K$ and $T$.
\end{tm}
\begin{proof}
Proposition \ref{solb} implies that there exists some constant $R>0$ depending on $\omega$, $\rho$, $\theta$, $\eta$, $\nu$ and $T$, such that $|\pi(t_{k-1},Y^{h}_{k-1},x^{h})_{t}|\leq R$ for all $h$, $k$, $t$. Based on this result, we know that the localization skill is applicable.
Combining Theorem \ref{WZk} and Theorem \ref{midrate}, we complete the proof.
\end{proof}

\section{Numerical experiments}\label{sec6}
In this section, we illustrate our theoretical results by two examples. In the first example, we verify the theoretical convergence rate given in Corollary \ref{RKrate}. In the second one, we consider the Kubo oscillator and compare the performance of the midpoint scheme, the simplified step-$2$ Euler scheme and the simplified step-$3$ Euler scheme.

\subsection{Example 1}
We consider a rough Hamiltonian system in $\R^{2}$ driven by a two-dimensional fBm with independent components and suppose
\begin{align*}
H_{0}(P,Q)=\sin(P)\cos(Q),\ H_{1}(P,Q)=\cos(P),\ H_{2}(P,Q)=\sin(Q).
\end{align*}
More precisely, the corresponding equations are
\begin{equation*}
\begin{split}
dP&=\sin(P)\sin(Q)dt-\cos(Q)dX^{2},\quad P_{0}=p,\\
dQ&=\cos(P)\cos(Q)dt-\sin(P)dX^{1},\quad Q_{0}=q.
\end{split}
\end{equation*}

{\it Convergence rate}. Since the vector fields are bounded with bounded derivatives, the theoretical convergence rates are almost $0.3,0.2,0.1$ for both Method \uppercase\expandafter{\romannumeral1} and Method \uppercase\expandafter{\romannumeral2} in Corollary \ref{RKrate} when the Hurst parameter $H$ are $0.4,0.35,0.3$ respectively. Figure \ref{Hboundorder} shows the maximum error in the discretisation points,  $\max_{k=1,...,N}|Y^{h}_{k}(\omega)-Y_{t_{k}}(\omega)|$, which we call the pathwise maximum error for short, for Method \uppercase\expandafter{\romannumeral1}. Pictures in the same row are from three sample paths. Numerical results are in acceptable accordance with theoretical analysis. In addition, the fact that the error becomes larger as $H$ increases indicates the influence of the roughness of the noise.

\subsection{Example 2}
For the linear case, we consider the Kubo oscillator driven by a three-dimensional fBm with independent components and $H=0.4$:
\begin{equation}\label{Kubo}
\begin{split}
dP=&-Qdt-\epsilon\sum^{d}_{i=1}Qd X^{i},\quad P_{0}=p,\\
dQ=&Pdt+\epsilon\sum^{d}_{i=1}Pd X^{i},\quad Q_{0}=q.
\end{split}
\end{equation}

We have the expression of the exact solution
$$P=p\cos(\sigma_{0}t+\epsilon\sum^{d}_{i=1}\sigma_{i}X_{t}^{i})-q\sin(\sigma_{0}t+\epsilon\sum^{d}_{i=1}\sigma_{i}X_{t}^{i}),$$
$$Q=q\cos(\sigma_{0}t+\epsilon\sum^{d}_{i=1}\sigma_{i}X_{t}^{i})+p\sin(\sigma_{0}t+\epsilon\sum^{d}_{i=1}\sigma_{i}X_{t}^{i}).$$

\begin{figure}[H]
\centering
 \subfigure[$H=0.4$]{
\begin{minipage}[t]{0.3\linewidth}
  \includegraphics[height=4.7cm,width=4.7cm]{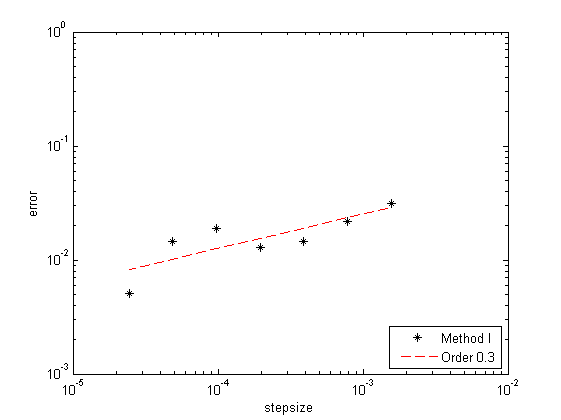}
 \end{minipage}
 }
  \subfigure[$H=0.4$]{
\begin{minipage}[t]{0.3\linewidth}
  \includegraphics[height=4.7cm,width=4.7cm]{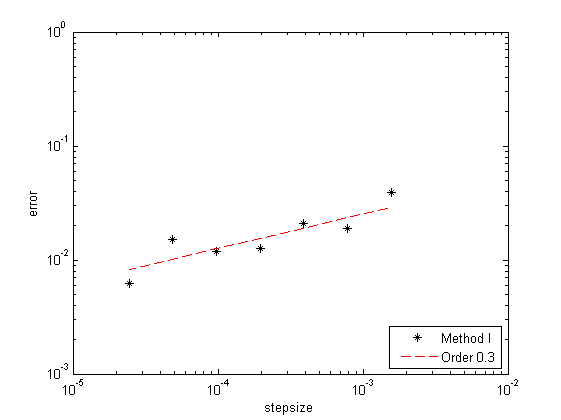}
  \end{minipage}
  }
 \subfigure[$H=0.4$]{
\begin{minipage}[t]{0.3\linewidth}
  \includegraphics[height=4.7cm,width=4.7cm]{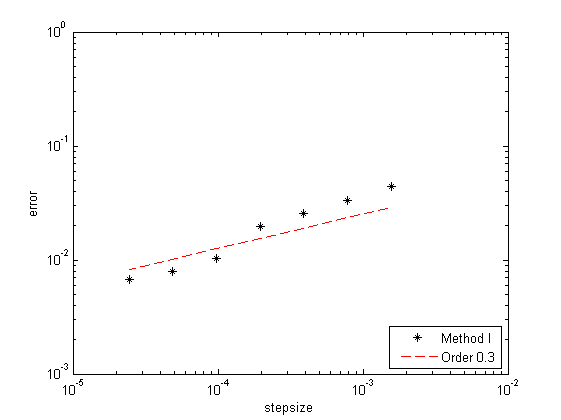}
 \end{minipage}
 }
\centering
 \subfigure[$H=0.35$]{
\begin{minipage}[t]{0.3\linewidth}
  \includegraphics[height=4.7cm,width=4.7cm]{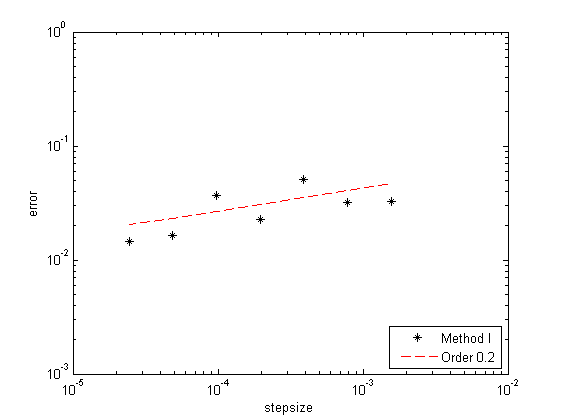}
 \end{minipage}
 }
 \subfigure[$H=0.35$]{
\begin{minipage}[t]{0.3\linewidth}
  \includegraphics[height=4.7cm,width=4.7cm]{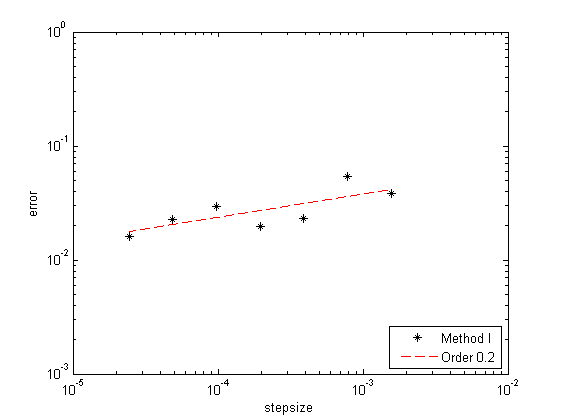}
 \end{minipage}
 }
 \subfigure[$H=0.35$]{
\begin{minipage}[t]{0.3\linewidth}
  \includegraphics[height=4.7cm,width=4.7cm]{H035p3.png}
 \end{minipage}
 }
\centering
 \subfigure[$H=0.3$]{
\begin{minipage}[t]{0.3\linewidth}
  \includegraphics[height=4.7cm,width=4.7cm]{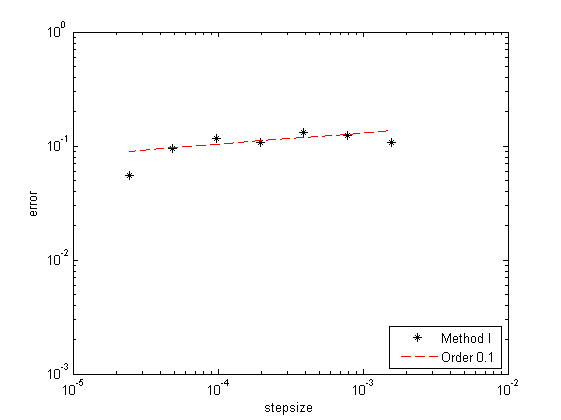}
 \end{minipage}
 }
 \subfigure[$H=0.3$]{
\begin{minipage}[t]{0.3\linewidth}
  \includegraphics[height=4.7cm,width=4.7cm]{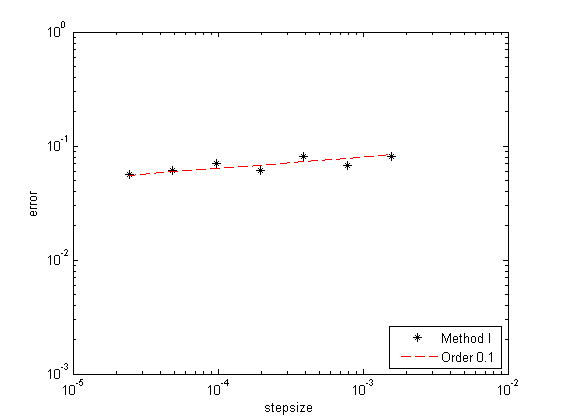}
 \end{minipage}
 }
 \subfigure[$H=0.3$]{
\begin{minipage}[t]{0.3\linewidth}
  \includegraphics[height=4.7cm,width=4.7cm]{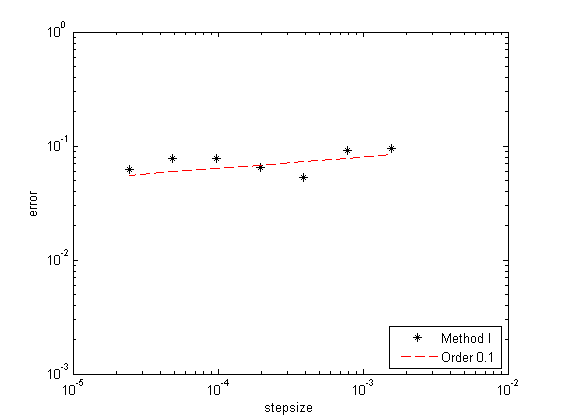}
 \end{minipage}
 }
\caption{Pathwise maximum error vs. step size for three sample paths with $p=1$, $q=2$ and $T=0.1$.}\label{Hboundorder}
\end{figure}

Denoting $Y=(P,Q)^\top\in\R^{2}$ and $X^{0}:=t$, we have that (\ref{Kubo}) is equivalent to $dY=\sum^{d}_{i=0}A_{i}YdX^{i}$, where $A_{0}$=$\left( \begin{array}{cc}
0 & -1 \\
1 &0\end{array} \right)$ and $A_{i}$=$\left( \begin{array}{cc}
0 & -\epsilon \\
\epsilon &0\end{array} \right)$, $i=1,2,3$. Since $A_{i}$ are all skew symmetric, the midpoint scheme is solvable. We compare it with the simplified step-$2$ and step-$3$ Euler schemes in the following experiments.

\begin{figure}[H]
\centering
  \includegraphics[height=7.5cm,width=8.5cm]{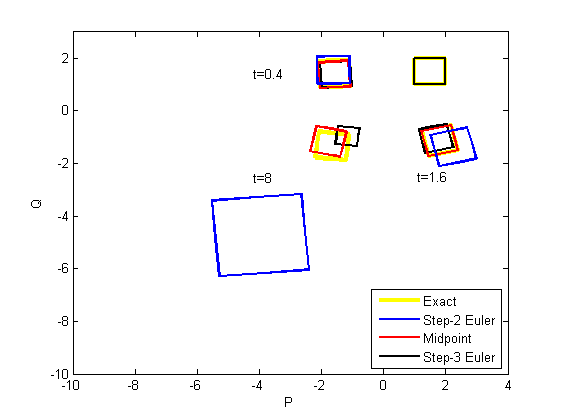}
\caption{Evolution of domains in the phase plane for one sample path with $\epsilon=1.5$, $T=10$ and $h=0.002$.}\label{sympic}
\end{figure}
\begin{figure}[H]
\centering
  \subfigure[]{
\begin{minipage}[t]{0.3\linewidth}
  \includegraphics[height=4.7cm,width=4.7cm]{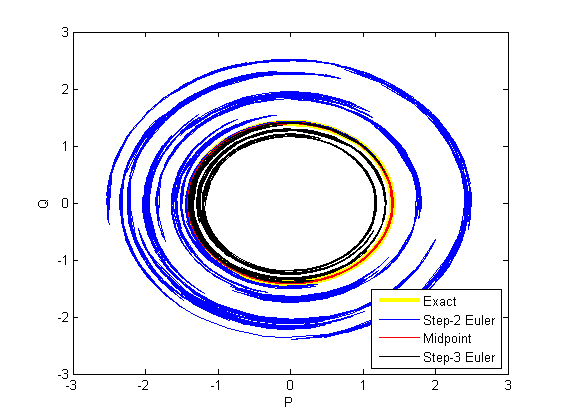}
  \end{minipage}
  }
 \subfigure[]{
\begin{minipage}[t]{0.3\linewidth}
  \includegraphics[height=4.7cm,width=4.7cm]{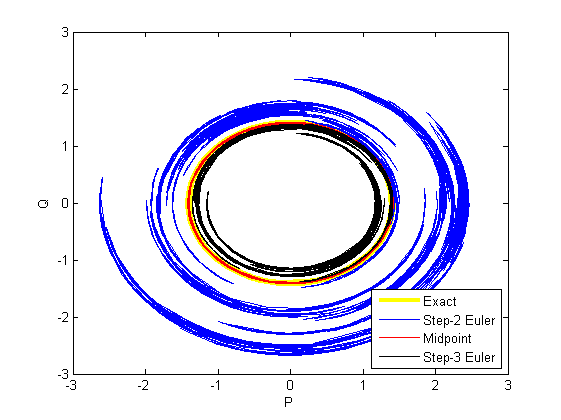}
 \end{minipage}
 }
 \subfigure[]{
\begin{minipage}[t]{0.3\linewidth}
  \includegraphics[height=4.7cm,width=4.7cm]{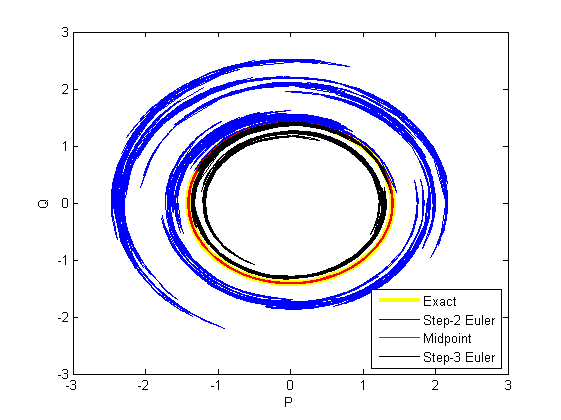}
 \end{minipage}
 }
\caption{Phase trajectory for three sample paths with $p=1$, $q=1$, $\epsilon=1$, $T=10$ and $h=10^{-12}$.}\label{Kubophase}
\end{figure}

{\it Evolution of domains}. Figure \ref{sympic} shows the evolution of domains in the phase plane for one sample path. The initial domain is a square with four corners at $(1,1)$, $(2,1)$, $(2,2)$ and $(1,2)$. Images at $t=0.4,1.6,8$, are presented under the exact mapping and the three numerical methods. One can observe that the images are still squares. The exact mapping is area preserving, which is equivalent to that it preserves the symplectic structure of the system. This property is inherited by the midpoint scheme, since the area of red squares are the same. However, the area of images under the simplified step-$2$ and step-$3$ Euler schemes increase and decrease respectively, which proves that they are not symplectic methods.

\begin{figure}[H]
\centering
 \subfigure[$\epsilon=1$, $T=1$]{
\begin{minipage}[t]{0.3\linewidth}
  \includegraphics[height=4.7cm,width=4.7cm]{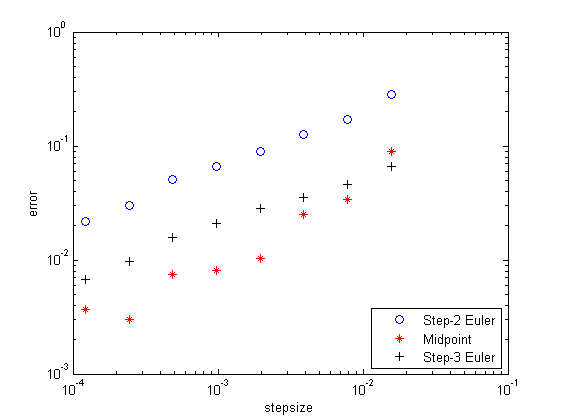}
 \end{minipage}
 }
 \subfigure[$\epsilon=1$, $T=1$]{
\begin{minipage}[t]{0.3\linewidth}
  \includegraphics[height=4.7cm,width=4.7cm]{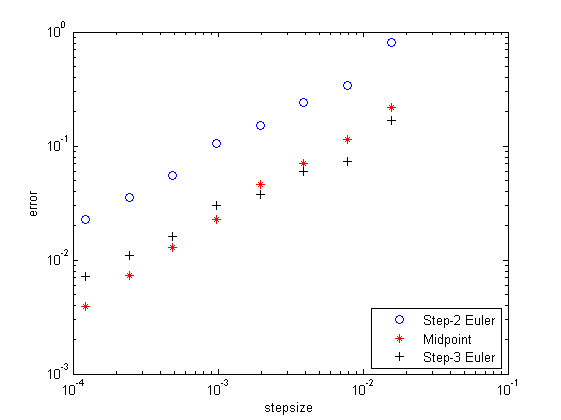}
 \end{minipage}
 }
 \subfigure[$\epsilon=1$, $T=1$]{
\begin{minipage}[t]{0.3\linewidth}
  \includegraphics[height=4.7cm,width=4.7cm]{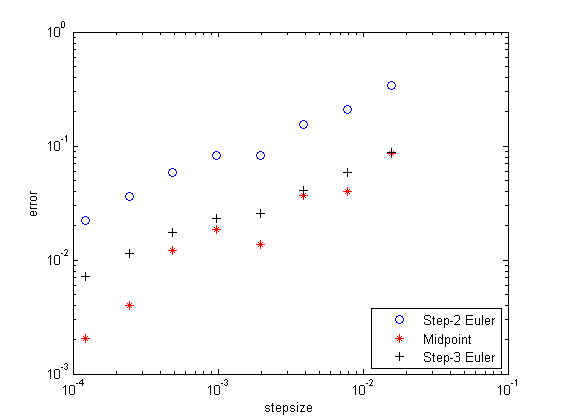}
 \end{minipage}
 }
  \subfigure[$\epsilon=2$, $T=1$]{
\begin{minipage}[t]{0.3\linewidth}
  \includegraphics[height=4.7cm,width=4.7cm]{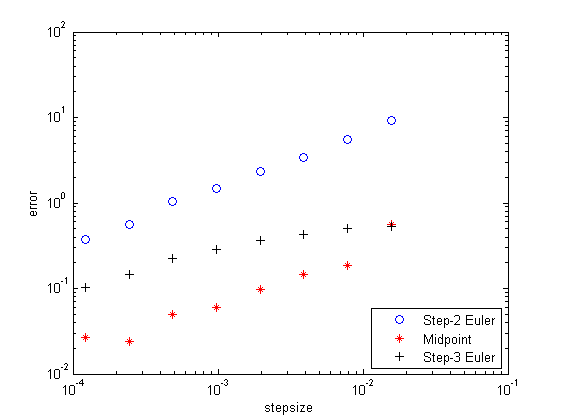}
  \end{minipage}
  }
 \subfigure[$\epsilon=2$, $T=1$]{
\begin{minipage}[t]{0.3\linewidth}
  \includegraphics[height=4.7cm,width=4.7cm]{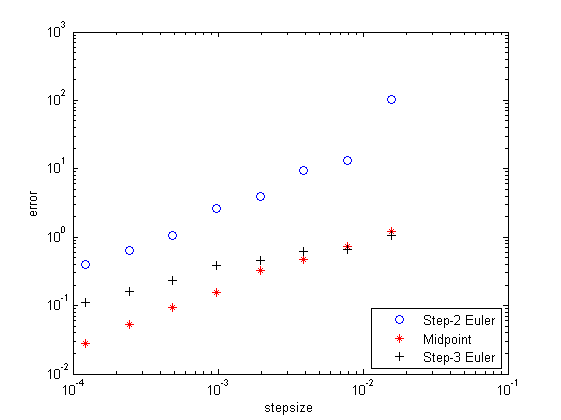}
 \end{minipage}
 }
 \subfigure[$\epsilon=2$, $T=1$]{
\begin{minipage}[t]{0.3\linewidth}
  \includegraphics[height=4.7cm,width=4.7cm]{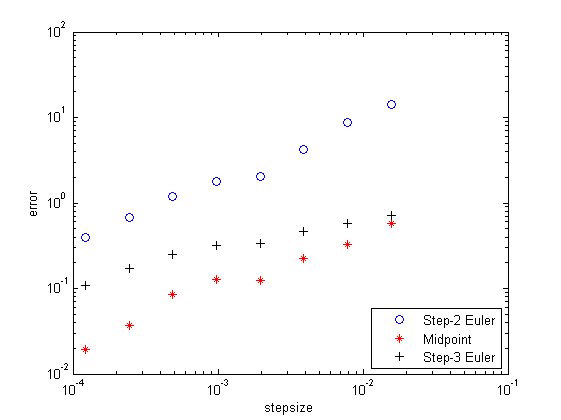}
 \end{minipage}
 }
  \subfigure[$\epsilon=1$, $T=10$]{
\begin{minipage}[t]{0.3\linewidth}
  \includegraphics[height=4.7cm,width=4.7cm]{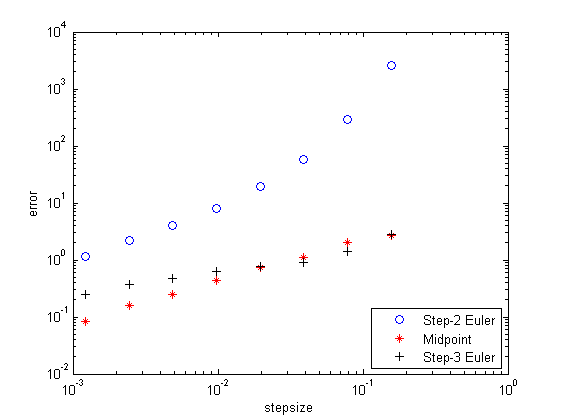}
  \end{minipage}
  }
 \subfigure[$\epsilon=1$, $T=10$]{
\begin{minipage}[t]{0.3\linewidth}
  \includegraphics[height=4.7cm,width=4.7cm]{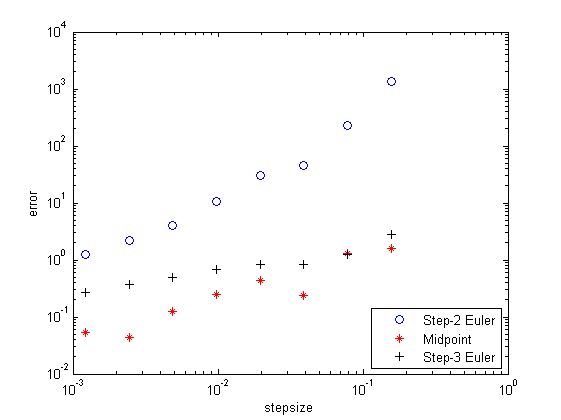}
 \end{minipage}
 }
 \subfigure[$\epsilon=1$, $T=10$]{
\begin{minipage}[t]{0.3\linewidth}
  \includegraphics[height=4.7cm,width=4.7cm]{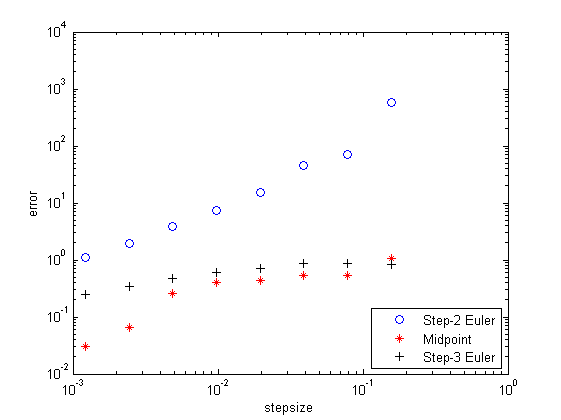}
 \end{minipage}
 }
\caption{Pathwise maximum error vs. step size for three sample paths with $p=1$ and $q=1$.}\label{Kuboorder}
\end{figure}

{\it Conservation of quadratic invariant}. For the initial value $(1,1)$, the exact solution is on the circle with center at the origin and radius $r=\sqrt{2}$, which implies another invariant of this system (see also \cite[Example 1]{MLMC16}). We present the phase trajectory for three sample paths in Figure \ref{Kubophase}. The midpoint scheme preserves this property such that its solutions are always on that circle, which is proved in Proposition \ref{solb}. The phase trajectory of the simplified step-$2$ Euler scheme deviates from the exact one a lot and that of the simplified step-$3$ Euler scheme shrinks gradually. Figure \ref{sympic} and Figure \ref{Kubophase} verify that the midpoint scheme, as a symplectic Runge--Kutta method, preserves the invariants of the system (\ref{Kubo}) and show its long time behavior is better than another two methods.

{\it Stability for $\epsilon$ and $T$}. We investigate the influence of the size $\epsilon$ of the noise and the total time $T$ in Figure \ref{Kuboorder}. When $T=10$, the simplified step-$2$ scheme is unstable. The pathwise maximum error of the midpoint scheme is the smallest among the three methods. As $\epsilon$ and $T$ become larger, this advantage is more clearer. This result implies the stability of an implicit method.

\bibliographystyle{plain}
\bibliography{implicit}

\end{document}